\documentclass[article,oneside,oldfontcommands,a4paper,1จpt]{memoir}

\counterwithout{section}{chapter}
\setsecnumdepth{subsubsection}

\setsecheadstyle{\large\bf\raggedright}
\setsubsecheadstyle{\large\sl\raggedright}
\setsubsubsecheadstyle{\sl\raggedright}
\setsecnumformat{\csname the#1\endcsname.\;}

\usepackage[affil-it]{authblk}
\usepackage{indentfirst}
\usepackage{geometry}

\usepackage[numbers,sort&compress]{natbib}

\usepackage{amssymb,amsmath,mathrsfs}
\usepackage{amsthm,stmaryrd}
\usepackage[hypertexnames=false,colorlinks,linkcolor={blue},citecolor={blue}]{hyperref}
\usepackage{graphicx,float}
\usepackage{enumitem}
\usepackage[T1]{fontenc}

\usepackage{fancyhdr}
\numberwithin{equation}{section}

\newtheorem{thm}{Theorem}[section]
\newtheorem{cor}[thm]{Corollary}
\newtheorem{lem}[thm]{Lemma}
\newtheorem{prop}[thm]{Proposition}
\theoremstyle{definition}
\newtheorem{dfn}[thm]{Definition}

\newtheorem{ques}{Question}
\theoremstyle{remark}
\newtheorem*{rmk}{Remark}

\allowdisplaybreaks

\newcommand{\reDeclareMathOperator}[2]{\let#1\undefined \DeclareMathOperator{#1}{#2}}

\reDeclareMathOperator{\mod}{mod}

\DeclareMathOperator{\dom}{dom}
\DeclareMathOperator{\lev}{lev}

\DeclareMathOperator{\diam}{diam}
\reDeclareMathOperator{\supp}{supp}
\DeclareMathOperator*{\argmin}{arg\,min}
\reDeclareMathOperator{\Proj}{Proj}
\reDeclareMathOperator{\prox}{prox}
\DeclareMathOperator{\grp}{grp}

\DeclareMathOperator{\co}{co}
\DeclareMathOperator{\cl}{cl}

\renewcommand{\emptyset}{\varnothing}
\newcommand{\CAT}{\rm CAT}
\newcommand{\0}{\mathbf{0}}
\newcommand{\R}{\mathbb{R}}
\newcommand{\N}{\mathbb{N}}

\renewcommand{\to}{\rightarrow}

\newcommand{\dto}{\downarrow}
\newcommand{\tendsto}{\longrightarrow}

\newcommand{\abs}[1]{\left|#1\right|}
\newcommand{\pair}[2]{\langle#1,#2\rangle}

\renewcommand{\multimap}{\rightrightarrows}

\renewcommand{\Vec}[1]{\overrightarrow{#1}}

\makeatletter
\renewcommand*{\@fnsymbol}[1]{\ifcase#1\or*\else\@arabic{#1}\fi}
\makeatother

\title{Equilibrium Problems and Proximal Algorithms in Hadamard Spaces}

\author{Poom Kumam and Parin Chaipunya\thanks{Corresponding author.}
\vspace{-.15cm}\footnote{}}
\affil{\footnotesize
Department of Mathematics, Faculty of Science,\\
King Mongkut's University of Technology Thonburi, \\
126 Pracha Uthit Rd., Bang Mod, Thung Khru,\\
Bangkok 10140, Thailand.\\
Emails: poom.kum@kmutt.ac.th (P. Kumam),\\ parin.cha@mail.kmutt.ac.th (P. Chaipunya).
}

\date{}
\begin{document}
\maketitle \vspace{-1.2cm}
\thispagestyle{empty}
\begin{abstract}

In this paper, we consider the equilibrium problems and also their regularized problems under the setting of Hadamard spaces. The solution to the regularized problem is represented in terms of resolvent operators. As an essential machinery in the existence of an equilibrium, we first prove that the KKM principle is attained in general Hadamard spaces without assuming the compactness of the closed convex hull of a finite set. We construct the proximal algorithm based on this regularization and give convergence analysis adequately.

\medskip


\smallskip

\end{abstract}\normalsize


\section{Introduction}

Equilibrium problems were originally studied in \cite{MR1292380} as a unifying class of variational problems. Given a nonempty set $K$, and a bifunction $F : K \times K \to \R$. The equilibrium problem $EP(K,F)$ is formulated as follows: 
\[
\text{Find a point $\bar{x} \in K$ such that $F(\bar{x},y) \geq 0$ for every $y \in K$.}	\tag*{$EP(K,F)$}
\]
By assigning different settings to $F$, we can include, \emph{e.g.,} minimization, minimax inequalities, variational inequalities, and fixed point problems in the class of equilibrium problems. The set of all solutions of $EP(K,F)$ is denoted by $\mathcal{E}(K,F)$. Typical studies for $EP(K,F)$ are extensively carried out in Banach or Hilbert spaces, and recently in Hadamard manifolds.


Proximal algorithm is one of the most elementary method used in solving several classes of variational problems. The main idea of the method is to perturb (or regularize) the original problem into a simpler and more well-behaved problem, and solve for each steps the perturbed subproblems. This method was originally proposed under the setting of Hilbert spaces by Martinet \cite{MR0298899}, and was progressively developed by Rockafellar \cite{MR0410483}.

Proximal algorithms in Hadamard spaces were recently investigated. In particular, it was introduced for minimizing convex functionals by Ba\v{c}\'{a}k in \cite{MR3047087}, and were extended for solving variational inequalities in \cite{MR3411805} and for solving zeros of maximal monotone operators in \cite{MR3679017,MR3691338}.

Equilibrium problems in Hadamard manifold were studied by \cite{MR2869729} and were later extended to equilibrium problems for bifunctions defined on proximal pairs in \cite{MR3338819}. Both of the results rely on different variants of the KKM lemma (consult \cite{zbMATH02573692} for the original version), but the latter is strongly based on Brouwer's fixed point theorem in Hadamard manifolds (see \cite{MR1942573}). In fact, the Brouwer's theorem is proved in Hadamard spaces in \cite{MR2565775,MR2561730} and the KKM principle was subsequently proved in Hadamard space with the convex hull finite property (CHFP), \emph{i.e.,} every polytope is assumed to be compact (see also the convex hull property in \cite{MR1385525}).

Based on this KKM principle and under CHFP assumption, Kimura and Kishi \cite{Kimura-Kishi} studied the equilibrium problem and showed the well-definedness of the resolvent associated to a bifunction. In particular, they also proved that this resolvent is firmly nonspreading and has fixed point set identical to the equilibrium points. Finally, they apply a convergence theorem from \cite{MR3213144} to solve for an equilibrium.

In this paper, we show that the KKM principle can be extended to any Hadamard spaces without further assumptions. Using the KKM principle, we show the existence of an equilibrium of a bifunction under standard continuity, convexity, and compactness/coercivity assumptions. Then, we deduce several fundamental properties of the resolvent operator introduced by Kimura and Kishi \cite{Kimura-Kishi}. We also deliver some comparisons with resolvents of convex functionals and of monotone vector fields. We finally define the proximal algorithm by iterating resolvent operators and provide adequate convergence analysis of the algorithm. Apart from dropping the CHFP assumption, the contents presented in this paper provide a continuation from the works of Niculescu and Roven\c{t}a \cite{MR2565775,MR2561730} in the study of KKM principle, of Kimura and Kishi \cite{Kimura-Kishi} in the study of bifunctions, their resolvents, and proximal algorithms, and of the authors \cite{MR3691338} in the relationships between bifunctions and vector fields as well as their resolvents and convergence results.

The organization of this paper is as follows. The next section collects useful basic knowledges used in the rest of this paper. We also give in this section several auxiliary results that will be exploited to validate subsequent results. In particular, properties concerning the product $\pair{\cdot}{\cdot}$ and the weaker convergence notions are explained. Section \ref{sec:KKM} is devoted to the discussion on the KKM theory on Hadamard spaces. This contains the key tool for proving existence theorems, which will be found in Section \ref{sec:existence}. We also deduce here the dual problem in the sense of Minty and give relationships with the primal priblem. In Section \ref{sec:resolvents}, we introduce the resolvent operator corresponds to a bifunction. Properties of the resolvents are thoroughly deduced, especially with the nonexpansivity of the operator. This leads to the study of Section \ref{sec:convergence}, which contains the construction of the proximal algorithm and also its convergence analysis.

\section{Preliminaries and Auxiliaries}\label{sec:prelim}

We divide this section into several parts, describing each topics in brief details. This includes some basic definitions and natations up to the technical results that will be used in our main results in the next sections.

%

\subsection{Hadamard spaces}

A uniquely geodesic metric space $(X,\rho)$ is a $\CAT(0)$ space if each geodesic triangle in $X$ is at least as thin as its comparison triangle in Euclidean plane. A complete $\CAT(0)$ space is then called \emph{Hadamard space}. The following characterization of a $\CAT(0)$ space is useful.
\begin{prop}
For a geodesic metric space $(X,\rho)$, the following conditions are all equivalent:
\begin{enumerate}[label=(\roman*)]
\item $X$ is $\CAT(0)$.
\item (\cite{MR2460066}) For any $x,u,v \in X$ and $\lambda \in [0,1]$, the following inequality holds:
\[\label{eqn:CN}
\rho^{2}(x,\gamma_{u,v}(\lambda)) \leq (1-\lambda)\rho^{2}(x,u) + \lambda\rho^{2}(x,v) - \lambda(1-\lambda)\rho^{2}(u,v). \tag{CN}
\]
\item (\cite{MR2390077}) For any $x,y,u,v \in X$, the following inequality holds:
\begin{equation}\label{prop:ql-ineq}
\rho^{2}(x,v) + \rho^{2}(y,u) \leq \rho^{2}(x,u) + \rho^{2}(y,v) + 2\rho(x,y)\rho(u,v).
\end{equation}
\end{enumerate}
\end{prop}
\begin{dfn}
A convex set $K \subset X$ is said to be \emph{flat} if the \eqref{eqn:CN} inequality holds as an equality for each $x,u,v \in K$.
\end{dfn}

Let $(X,\rho)$ be an Hadamard space. For each $x,y \in X$, we write $\gamma_{x,y} : [0,1] \to X$ to denote the normalized geodesic joining $x$ and $y$, \emph{i.e.,} $\gamma_{x,y} (0) = x$, $\gamma_{x,y} (1) = y$, and $\rho(\gamma_{x,y}(t),\gamma_{x,y}(t')) = \rho(x,y)\abs{t-t'}$ for all $t,t' \in [0,1]$. We also adopt the following notations: $\llbracket x,y \rrbracket := \{\gamma_{x,y}(t) ; t \in [0,1]\}$ and $(1-t)x \oplus t y := \gamma_{x,y}(t)$ for $t \in [0,1]$. A subset $K \subset X$ is said to be \emph{convex} if $\llbracket x,y \rrbracket \subset K$ for any $x,y \in K$, and for a set $E \subset X$ we write $\co(E)$ to denote the smallest convex set that contains $E$. Certainly, we call $\co(E)$ the \emph{convex hull} of $E$. A function $h : K \to \R$, with $K$ being convex, is called \emph{convex} (resp., \emph{quasi-convex}) if $h\circ \gamma_{x,y} : [0,1] \to \R$ is convex (resp., \emph{quasi-convex}) for any $x,y \in K$. Moreover, $h$ is called \emph{concave} (resp., \emph{quasi-concave}) if $-h$ is convex (resp., \emph{quasi-convex}).

If $(X,d)$ and $(X',d')$ are two Hadamard space, then the product $X \times X'$ is also an Hadamard space with the metric given by
\[
\rho\left((x,x'),(y,y')\right) := \left(d^{2}(x,y) + d'^{2}(x',y')\right)^{\frac{1}{2}}, \quad \forall (x,x'),(y,y') \in X \times X'.
\]

Unless otherwise stated, always assume throughout this paper that $(X,\rho)$ is an Hadamard space and $K \subset X$ is nonempty, closed, and convex. Any product space is also to be understood to be an Hadamard space in the sense described above.

\subsection{Dual space and Tangent spaces}

The concept of a dual space of $X$ was introduced in \cite{MR2680038}. Here, we recall such a construction in brief details. Let us write $\Vec{xy} := (x,y) \in X^{2}$. The \emph{quasilinearization} (see \cite{MR1637650,MR2390077}) on $X^{2}$ is the product $\pair{\cdot}{\cdot} : X^{2} \times X^{2} \to \R$ defined by
\[
\pair{\Vec{uv}}{\Vec{xy}} := \cfrac{1}{2}[\rho^{2}(u,y) + \rho^{2}(v,x) - \rho^{2}(u,x) - \rho^{2}(v,y)], \quad \forall \Vec{uv},\Vec{xy} \in X^{2}.
\]
For each $t \in \R$ and $\Vec{uv} \in X^{2}$, we define the Lipschitzian function $\Theta(t;\Vec{uv}) : X \to \R$ by
\[
\Theta(t;\Vec{uv})(x) := t\pair{\Vec{uv}}{\Vec{ux}}, \quad \forall x \in X.
\]
Recall that for a Lipschitzian function $h : X \to \R$, the Lipschitzian constant of $h$ is the quantity $L(h) := \inf\{\frac{\abs{h(x)-h(y)}}{\rho(x,y)} \,|\, x,y \in X,\, x \neq y\}$. Applying the Lipschizian constant to the differences of $\Theta$'s, we can define a pseudometric $\tilde D$ on $\R \times X^{2}$ by:
\[
\tilde{D}\left((s,\Vec{uv}),(t,\Vec{xy})\right) := L(\Theta(s;\Vec{uv}) - \Theta(t;\Vec{xy})), \quad \forall (s,\Vec{uv}),(t,\Vec{xy}) \in \R \times X^{2},
\]
which naturally gives rise to the following equivalence relation
\[
(s,\Vec{uv}) \sim (t,\Vec{xy}) \iff \tilde{D}\left((s,\Vec{uv}),(t,\Vec{xy})\right) = 0, \quad \forall (s,\Vec{uv}),(t,\Vec{xy}) \in \R \times X^{2}.
\]
The quotient space $X^{\ast} := \R\times X^{2}/\sim$ with a metric $D$ given by 
\[
D([(s,\Vec{uv})]_{\sim},[(t,\Vec{xy})]_{\sim}) = \tilde{D}\left((s,\Vec{uv}),(t,\Vec{xy})\right), \quad \forall [(s,\Vec{uv})]_{\sim},[(t,\Vec{xy})]_{\sim} \in X^{\ast},
\]
is called the \emph{dual space} of $X$. For simplicity, we adopt the notation $s\Vec{uv} := [(s,\Vec{uv})]_{\sim}$. Moreover, we write $\0 := s\Vec{uu} = 0\Vec{uv}$ for $s \in \R$ and $u,v \in X$.

In this paper, we restrict to particular subspaces of $X^{\ast}$. For any given $p \in X$, the \emph{tangent space} of $X$ at $p$ is given by $T_{p}X := \{s\Vec{py} \,|\, s \geq 0, \, y \in X\}$.

The following properties of the quasilinearization are fundamental.

\begin{lem}\label{prop:twopair}
For every $u,v,z \in X$, the following properties hold:
\begin{enumerate}[label=(\roman*)]
\item\label{prop:twopair1} $\pair{\Vec{uv}}{\Vec{uz}} + \pair{\Vec{vu}}{\Vec{vz}} = \rho^{2}(u,v) \geq 0$.
\item\label{prop:ontheway} If $\lambda \in [0,1]$ and $x = \gamma_{z,u}(\lambda)$, it holds that $\lambda\pair{\Vec{zu}}{\Vec{zv}} \leq \pair{\Vec{zx}}{\Vec{zv}}$. In addition, if a convex set $K \subset X$ is flat, then the above inequality becomes equality for $u,v,z \in K$.
\item\label{prop:affine} If a convex set $K \subset X$ is flat and $z,v \in K$, then $u \in K \mapsto \pair{\Vec{zu}}{\Vec{zv}}$ is affine on $K$, i.e., it is both convex and concave on $K$.
\end{enumerate}
\end{lem}
\begin{proof}
The proof for \ref{prop:twopair1} is trivial. For \ref{prop:ontheway}, the \eqref{eqn:CN} inequality gives
\begin{align*}
2\pair{\Vec{zx}}{\Vec{zv}} &= \rho^{2}(z,x) + \rho^{2}(z,v) - \rho^{2}(x,v) \\
&= \lambda^{2}\rho^{2}(z,u) + \rho^{2}(z,v) - \rho^{2}(\gamma_{z,u}(\lambda),v) \\
&\geq \lambda^{2}\rho^{2}(z,u) + \rho^{2}(z,v) \\&\quad\quad - \left[(1-\lambda)\rho^{2}(z,v) + \lambda\rho^{2}(u,v) - \lambda(1-\lambda)\rho^{2}(z,u)\right]\\
&= \lambda\rho^{2}(z,u) + \lambda\rho^{2}(z,v) - \lambda\rho^{2}(u,v)\\
&= 2\lambda\pair{\Vec{zu}}{\Vec{zv}},
\end{align*}\normalsize
which proves the first assertion. Now, if a convex set $K \subset X$ is flat, then the \eqref{eqn:CN} inequatily becomes equality for $u,v,z \in K$, and the desired result is straightforward from the above proofline.

Next, let us show \ref{prop:affine}. Let $p,q \in K$ and $\lambda \in [0,1]$. It follows that
\begin{align*}
2\pair{\Vec{z\gamma_{p,q}(\lambda)}}{\Vec{zv}} &= \rho^{2}(z,\gamma_{p,q}(\lambda)) + \rho^{2}(z,v) - \rho^{2}(\gamma_{p,q}(\lambda),v) \\
&= \left[(1-\lambda)\rho^{2}(z,p) + \rho^{2}(z,q) - \lambda(1-\lambda)\rho^{2}(p,q)\right] + \rho^{2}(z,v) \\
&\qquad- \left[(1-\lambda)\rho^{2}(v,p) + \lambda\rho^{2}(v,q) - \lambda(1-\lambda)\rho^{2}(p,q)\right] \\
&= (1-\lambda)\left[\rho^{2}(z,p) + \rho^{2}(z,v) - \rho^{2}(p,v)\right] \\
&\qquad + \lambda\left[\rho^{2}(z,q) + \rho^{2}(z,v) - \rho^{2}(q,v)\right] \\
&= 2\cdot \left[(1-\lambda)\pair{\Vec{zp}}{\Vec{zv}} + \lambda\pair{\Vec{zq}}{\Vec{zv}}\right].
\end{align*}
The proof is thus completed.
\end{proof}

\subsection{Modes of convergence}

Convergence in the metric topology is known to be irrelevant in some situations, especially in the study of numerical algorithms in infinite dimensional spaces. In this subsection, we recall two alternative modes of convergence for bounded sequences, namely the $\Delta$- and $w$-convergences. Both of the concepts are identical to weak convergence in Hilbert spaces.

Let us start with the $\Delta$-convergence. Suppose that $(x^{k}) \subset X$ be a bounded sequence, and define a function $r(\cdot;(x^{k})) : X \to [0,\infty)$ by
\[
r(x;(x^{k})) := \limsup_{k \tendsto \infty} \rho(x,x^{k}), \quad \forall x \in X.
\]
The minimizer of this function is known to exists and is unique (see \cite{MR2232680}). Following \cite{MR2416076} (see also \cite{MR0423139}), a bounded sequence $(x^{k})$ is said to be \emph{$\Delta$-convergent} to a point $\bar{x} \in X$ if $\bar{x} = \argmin_{x \in X} r(x;(u^{k}))$ for any subsequence $(u^{k}) \subset (x^{k})$. In this case, $\bar{x}$ is called the $\Delta$-limit of $(x^{k})$. Recall that a bounded sequence is $\Delta$-convergent to at most one point.

The topology that generates the $\Delta$-convergence is unknown in general. However, we still adopt the topological-like notions such as $\Delta$-accumulation points, $\Delta$-closed sets, or $\Delta$-continuity. For instance, a point $u \in X$ is a \emph{$\Delta$-accumulation point} of the sequence $(x^{k}) \subset X$ if it contains a subsequence that is $\Delta$-convergent to $u$. A set $K \subset X$ is called \emph{$\Delta$-closed} if each $\Delta$-convergent sequence in $K$ has its $\Delta$-limit in $K$. A function $f : X \to \R$ is called \emph{$\Delta$-upper semicontinuous} (briefly \emph{$\Delta$-usc}) if its epigraph is $\Delta$-closed in $X \times \R$.

%

The following proposition gives two most important properties regarding the $\Delta$-convergence that are required in our main theorems.
\begin{prop}\label{prop:deltaconvergence}
Suppose that $(x^{k}) \subset X$ is bounded. Then, the following properties hold:
\begin{enumerate}[label=(\roman*)]
\item\label{cdn:deltacompact} (\cite{MR0423139}) $(x^{k})$ has a $\Delta$-convergent subsequence .
\item\label{cdn:deltaproductcharacterization} (\cite{MR3003694}) $(x^{k})$ is $\Delta$-convergent to $\bar{x} \in X$ if and only if $\limsup_{k} \pair{S}{\Vec{\bar{x}x^{k}}} \leq 0$ for any $S \in T_{\bar{x}}X$ .
\end{enumerate}
\end{prop}

Next, let us turn to the notion of $w$-convergence as introduced in \cite{MR3003694}. A bounded sequence $(x^{k}) \subset X$ is said to be \emph{$w$-convergence} to a point $\bar{x} \in X$ if $\lim_{k} \pair{S}{\Vec{\bar{x}x^{k}}} = 0$ for any $S \in T_{\bar{x}}X$. With Proposition \ref{prop:deltaconvergence}, we can see immediately that $w$-convergence implies $\Delta$-convergence. As was noted in \cite{MR3003694,MR2831141}, a bounded sequence does not necessarily have a $w$-convergent subsequence. This motivates the definition of reflexivity in Hadamard spaces, i.e., a convex set $K \subset X$ is said to be \emph{reflexive} if each bounded sequence in $K$ contains a $w$-convergent subsequence. However, it turns out that the reflexivity of $K$ implies the equivalence between the $\Delta$- and weak convergences. Still, the reflexivity can be useful in obtaining sharper estimates in some situations (see e.g. the proof of Lemma \ref{lem:fB-skewedUSC}).

\begin{prop}[\cite{MR3241330}]\label{lem:closed+convex=Deltaclosed}
A convex set $K \subset X$ is closed if and only if it is $\Delta$-closed.
\end{prop}

\begin{prop}\label{prop:delta-iff-weak}
Suppose that a convex set $K \subset X$ is reflexive and a bounded sequence $(x^{k}) \subset K$ is $\Delta$-convergent. Then, it is $w$-convergent.
\end{prop}
\begin{proof}
Suppose that $(x^{k})$ is $\Delta$-convergent to $\bar{x} \in X$. Let us assume to the contrary that $(x^{k})$ is not weakly convergent. Equivalently, there must exists $\gamma_{0} \in T_{\bar{x}}X$ such that
\begin{equation}\label{eqn:liminf<limsup}
\liminf_{k \tendsto \infty} \pair{\Vec{\bar{x}x^{k}}}{\gamma_{0}} < \limsup_{k \tendsto \infty} \pair{\Vec{\bar{x}x^{k}}}{\gamma_{0}} \leq 0.
\end{equation}
Suppose that $(x^{k_{j}}) \subset (x^{k})$ is a subsequence such that
\[
\lim_{j \tendsto \infty} \pair{\Vec{\bar{x}x^{k_{j}}}}{\gamma_{0}} = \liminf_{k \tendsto \infty} \pair{\Vec{\bar{x}x^{k}}}{\gamma_{0}}.
\]
Take into account the inequality \eqref{eqn:liminf<limsup}, we get $\lim_{j} \pair{\Vec{\bar{x}x^{k_{j}}}}{\gamma_{0}} < 0$, which prevents $(x^{k})$ from having a weakly convergent subsequence. Since $(x^{k})$ is bounded in $K$, this violates the reflexivity of $K$. Therefore, $(x^{k})$ must be weakly convergent to $\bar{x}$.
\end{proof}

\subsubsection{Fej\'er convergence}

Lastly, the notion of Fej\'er convergence is essential and will play a central role in the proof of our main convergence theorems. This notion encapsulates the improvement at each iteration of some approximate sequence towards a solution set.

\begin{dfn}
A sequence $(x^{k}) \subset X$ is said to be \emph{Fej\'er convergent} with respect to a nonempty set $V \subset X$ if for each $x \in V$, we have $\rho(x^{k+1},x) \leq \rho(x^{k},x)$ for all large $k \in \N$.
\end{dfn}

\begin{prop}[\cite{MR3691338}]\label{prop:Fejer}
Suppose that $(x^{k}) \subset X$ is Fej\'er convergent to a nonempty set $V \subset X$. Then, the following are true:
\begin{enumerate}[label=(\roman*)]
\item $(x^{k})$ is bounded.
\item $(\rho(x,x^{k}))$ converges for any $x \in V$.
\item If every $\Delta$-accumulation point lies within $V$, then $(x^{k})$ is $\Delta$-convergent to an element in $V$.
\end{enumerate}
\end{prop}

\section{The KKM Principle}\label{sec:KKM}

The KKM principle was initiated in \cite{zbMATH02573692}, and was successfully extended into topological vector spaces by Fan in \cite{MR0131268}. Further extensions into Hadamard manifold was discussed in \cite{MR2869729,MR3338819}. The results can be generalized instantly also into Hadamard spaces with fixed point property for continuous mappings defined on a convex hull of finite points \cite{MR2561730}. Here we prove the KKM principle without using such condition on the space. First, let us recall the original statement of \cite{zbMATH02573692} and another additional result.

\begin{lem}[\cite{zbMATH02573692}]
Let $C_{1},\dots,C_{m}$ be closed subsets of the standard $(m-1)$-simplex $\sigma$. If $\co(\{x_{i} \;;\; i \in I\}) \subset \bigcup_{i \in I} C_{i}$ for each $I \subset \{1,\dots,m\}$, then the intersection $\bigcap_{j=1}^{m} C_{j}$ is nonempty.
\end{lem}

\begin{lem}
Suppose that $(x^{k})$ is a sequence in $X$ and define the following sequence of sets by induction:
\begin{equation}\label{eqn:Dj}
\left\{
\begin{array}{l}
D_{1} := \{x_{1}\}, \medskip\\
D_{j} := \{z \in \llbracket x^{j},y \rrbracket \;;\; y \in D_{j-1}\}, \quad\text{for $j = 2,3,\dots$}
\end{array}
\right.
\end{equation}
Then, $D_{j}$ is compact for all $j \in \N$.
\end{lem}
\begin{proof}
We shall prove the statement by using mathematical induction. It is obvious that $D_{1}$ is compact. Now, we enter the inductive step by assuming $D_{j}$ is compact and show that $D_{j+1}$ must be compact.

Suppose that $(u^{k})$ is an arbitrary sequence in $D_{j+1}$. By definition, there correspond sequences $(v^{k})$ in $D_{j}$ and $(t_{k})$ in $[0,1]$ such that
\[
u^{k} = \gamma_{x^{j+1},v^{k}}(t_{k}), \quad \forall k \in \N.
\]
Since both $D_{j}$ and $[0,1]$ are compact, we may find convergent subsequences $(v^{k_{i}})$ of $(v^{k})$ and $(t_{k_{i}})$ of $(t_{k})$ with limits $\bar{v} \in D_{j}$ and $\bar{t} \in [0,1]$, respectively. Set $\bar{u} := \gamma_{x^{j+1},\bar{v}}(\bar{t})$, we now show that $(u^{k_{i}})$ is in fact convergent to $\bar{u}$.

Observe that
\begin{align*}
\rho(u^{k_{i}},\bar{u}) &= \rho(\gamma_{x^{j+1},v^{k_{i}}}(t_{k_{i}}),\gamma_{x^{j+1},\bar{v}}(\bar{t})) \\
&\leq \rho(\gamma_{x^{j+1},v^{k_{i}}}(t_{k_{i}}),\gamma_{x^{j+1},\bar{v}}(t_{k_{i}})) + \rho(\gamma_{x^{j+1},\bar{v}}(t_{k_{i}}),\gamma_{x^{j+1},\bar{v}}(\bar{t})) \\
&\leq t_{k_{i}} \rho(v^{k_{i}},\bar{v}) + \abs{t_{k_{i}} - \bar{t}} \rho(x^{j+1},\bar{v}).
\end{align*}
Passing $i \tendsto \infty$, we obtain from the above inequalities that $u^{k_{i}} \tendsto \bar{u}$. This guaratees the compactness of $D_{j+1}$, and the desired conclusion is thus proved.
\end{proof}

\begin{thm}[The KKM Principle]\label{lem:KKM}
Let $K \subset X$ be a closed convex set, $G : K \multimap K$ be a set-valued mapping with closed values. Suppose that for any finite subset $D := \{x_{1},\dots,x_{m}\} \subset K$, it holds the following inclusion:
\begin{equation}\label{eqn:KKMcondition}
\co(D) \subset G(D).
\end{equation}
Then, the family $\{G(x)\}_{x \in K}$ has the finite intersection property. Moreover, if $G(x_{0})$ is compact for some $x_{0} \in K$, then $\bigcap_{x \in K} G(x) \neq \emptyset$.
\end{thm}
\begin{proof}
Let $D := \{x_{1},\dots,x_{m}\}$ be an arbitrary finite subset of $K$. For each $j = 1,2,\dots,m$, define $D_{j}$ by using \eqref{eqn:Dj}. Also, set $D^{\ast} := \bigcup_{j=1}^{m} D_{j}$.

Let $\sigma := \langle e_{1},\dots,e_{m} \rangle$ be the standard $(m-1)$-simplex in $\R^{m}$. Suppose that $\lambda_{2} \in \langle e_{1},e_{2} \rangle$, then we can represent $\lambda_{2}$ with a scalar $s_{1} \in [0,1]$ such that $\lambda = (1-s_{1})e_{2} + s_{1}e_{1}$. Now, if $\lambda_{3} \in \langle e_{1},e_{2},e_{3} \rangle$, then $\lambda_{3} = (1-s_{2})e_{3} + s_{2}\lambda_{2}$ for some $\lambda_{2} \in \langle e_{1},e_{2} \rangle$. By the earlier fact, $\lambda_{2}$ is represented by some $s_{1} \in [0,1]$.  Hence, $\lambda_{3}$ can be represented by a $2$-dimensional vector $[s_{1} \ s_{2}]^{\top} \in [0,1]^{2}$. Likewise for $j = 3,\dots,m$, we can see that any $\lambda_{j} \in \langle e_{1},\dots,e_{j} \rangle$ is represented by a $(j-1)$-dimensional vector $[s_{1} \ \dots \ s_{j-1}]^{\top} \in [0,1]^{j-1}$. We shall adopt the notation $\lambda_{j} \equiv [s_{1} \ \dots \ s_{j-1}]^{\top}$ for the above representation.

Define a mapping $T : \sigma \to D^{\ast}$ by induction as follows: if $\lambda_{2} \equiv s_{1} \in \langle e_{1},e_{2} \rangle$, let $T(\lambda_{2}) := \gamma_{x_{2},x_{1}}(s_{1})$. For $j = 2,3,\dots,m$, if $\lambda_{j} \equiv [s_{1} \ \dots \ s_{j-1}]^{\top} \in \langle e_{1},\dots,e_{j} \rangle \setminus \langle e_{1},\dots,e_{j-1} \rangle$, then define $T(\lambda_{j}) := \gamma_{x_{j},T(\lambda_{j-1})}(s_{j-1})$, where $\lambda_{j-1} \equiv [s_{1} \ \dots \ s_{j-2}]^{\top} \in \langle e_{1},\dots,e_{j-1} \rangle$.

We now show that $T$ is continuous. Let $\lambda_{m},\mu_{m} \in \sigma$, and $\lambda \equiv [s_{1} \ \dots \ s_{m-1}]^{\top}$ and $\mu \equiv [t_{1} \ \dots \ t_{m-1}]^{\top}$ respectively. For simplicity, let $\lambda_{j} \equiv [s_{1} \ \dots \ s_{j-1}]^{\top}$ and $\mu_{j} \equiv [t_{1} \ \dots \ t_{j-1}]^{\top}$, for $j = 1,\dots,m$. Indeed, we have
\begin{align*}
\rho(T(\lambda_{m}),T(\mu_{m})) &= \rho(\gamma_{x_{m},T(\lambda_{m-1})}(s_{m-1}),\gamma_{x_{m},T(\mu_{m-1})}(t_{m-1})) \\
&\leq \rho(\gamma_{x_{m},T(\lambda_{m-1})}(s_{m-1}),\gamma_{x_{m},T(\lambda_{m-1})}(t_{m-1})) \\ &\quad+ \rho(\gamma_{x_{m},T(\lambda_{m-1})}(t_{m-1}),\gamma_{x_{m},T(\mu_{m-1})}(t_{m-1})) \\
&\leq \abs{s_{m-1} - t_{m-1}} \diam(D^{\ast}) + \rho(T(\lambda_{m-2}),T(\mu_{m-2})) \\
&\;\; \vdots \\
&\leq \sum_{i=1}^{m-1} \abs{s_{i} - t_{i}} \diam(D^{\ast}).
\end{align*}
This is sufficient to guarantee the continuity of $T$.

For each $j = 1,\dots,m$, define $E_{j} := T^{-1}(D^{\ast} \cap G_{i}(x_{j})) \subset \sigma$. By the continuity of $T$ and Lemma \ref{eqn:Dj}, we may see that $E_{j}$'s are closed sets. Suppose that $I \subset \{1,\dots,m\}$, $\lambda \in \co(\{e_{i} \;;\; i \in I\})$ and $\lambda \equiv [s_{1} \ \dots \ s_{m-1}]^{\top}$. Then, $s_{j} = 1$ if $j \not\in I$. By the definition of $T$, we have
\[
T(\lambda) \in \co(\{x_{i} \;;\; i \in I\}) \subset \bigcup_{i \in I} G(x_{i}).
\]
It follows that $T(\lambda) \in D^{\ast} \cap G(x_{i})$ for some $i \in I$. In other words, we have $\lambda \in E_{i}$ and therefore $\co(\{e_{i} \;;\; i \in I\}) \subset \bigcup_{i \in I} E_{i}$ for any $I \subset \{1,\dots,m\}$. By applying the original KKM covering lemma, we get the existence of $\lambda^{\ast} \in \bigcap_{j=1}^{m} E_{j}$. Hence, we get $T(\lambda^{\ast}) \in \bigcap_{j=1}^{m} G(x_{j})$. In fact, we have proved that the family of closed sets $\{G(x)\}_{x \in X}$ has the finite intersection property. The nonemptiness of the intersection $\bigcap_{x\in X} G(x) = \bigcap_{x\in X} (G(x)\cap G(x'))$ follows from the fact that $\{G(x)\cap G(x')\}_{x \in X}$ is a family of closed subsets with finite intersection property in a compact subspace $G(x')$.
\end{proof}

\section{Solutions of Equilibrium Problems}\label{sec:existence}

We mainly discuss in this section two topics. First, we deduce the solvability of the problem $EP(K,F)$ under standard (semi)continuity, convexity, and compactness/coercivity assumptions. Secondly, we introduce the dual problem to $EP(K,F)$ in the sense of Minty and show the relationship between the primal and dual problems, in terms of their solutions. Note that the knowledge of the dual problem is required in the subsequent sections.

\subsection{Existence theorems}

Based on the KKM principle in the previous section, we can show under some natural assumptions that the equilibrium problem $EP(K,F)$ is solvable.

\begin{thm}\label{thm:mainexistence}
Suppose that $K \subset X$ is closed convex, and $F : K \times K \to \R$ is a bifunction satisfying the following properties:
\begin{enumerate}[label=(A\arabic*)]
\item\label{asmp:1} $F(x,x) \geq 0$ for each $x \in K$.
\item\label{asmp:2} For every $x \in K$, the set $\{y \in K \;;\; F(x,y) < 0\}$ is convex.
\item\label{asmp:3} For every $y \in K$, the function $x \mapsto F(x,y)$ is usc.
\item\label{asmp:4} There exists a compact subset $L \subset K$ containing a point $y_{0} \in L$ such that $F(x,y_{0}) < 0$ whenever $x \in K\setminus L$.
\end{enumerate}
Then, the problem $EP(K,F)$ has a solution and the solution set $\mathcal{E}(K,F)$ is closed. In addition, if the function $x \mapsto F(x,y)$ is quasi-concave for every fixed $y \in K$, then $\mathcal{E}(K,F)$ is convex.
\end{thm}
\begin{proof}
Define a set-valued mapping $G : K \multimap K$ by
\[
G(y) := \overline{\lev}_{F(\cdot,y)}(0) = \{x \in K \;;\; F(x,y) \geq 0\}, \quad \forall y \in K.
\]
It is clear that $EP(K,F)$ has a solution if and only if $\bigcap_{y \in K} G(y) \neq \emptyset$. Thus, it is sufficient to show the nonemptiness of $\bigcap_{y \in K} G(y)$.

By \ref{asmp:1} and \ref{asmp:2}, we may see that $G$ has nonempty closed values at all $y \in K$. We will show now that $G$ satisfies the inclusion \eqref{eqn:KKMcondition}. Let us suppose to the contrary that there exist $y_{1},\dots,y_{m} \in K$ such that $\co(\{y_{1},\dots,y_{m}\}) \not\subset \bigcup_{j=1}^{m} G(y_{i})$. That is, there exists a point $y^{\ast} \in \co(\{y_{1},\dots,y_{m}\})$ such that $y^{\ast} \not\in G(y_{i})$ for any $i = 1,\dots,m$. It further implies that
\[
F(y^{\ast},y_{i}) < 0, \quad \forall i = 1,\dots,m.
\]
Moreover, we have for all $i = 1,\dots,m$, $y_{i} \in \{y \in K \;;\; F(y^{\ast},y) < 0\}$, which is a convex set by hypothesis. Since $y^{\ast} \in \co(\{y_{1},\dots,y_{m}\})$ and $\co(\{y_{1},\dots,y_{m}\})$ is the smallest convex set containing $y_{1},\dots,y_{m}$, we get
\[
y^{\ast} \in \co(\{y_{1},\dots,y_{m}\}) \subset \{y \in K \;;\; F(y^{\ast},y) < 0\},
\]
which further gives $F(y^{\ast},y^{\ast}) < 0$. This contradicts hypothesis \ref{asmp:1}, and therefore $G$ satisfies \eqref{eqn:KKMcondition}. On the other hand, hypothesis \ref{asmp:4} forces $G(y_{0}) \subset L$, which guarantees the compactness of $G(y_{0})$. Since $G$ satisfies every conditions of Lemma \ref{lem:KKM}, we get $\bigcap_{x \in X} G(x) \neq \emptyset$.

Next, suppose that $(x^{k}) \subset \mathcal{E}(K,F)$ is convergent to $\bar{x} \in K$, then \ref{asmp:3} gives
\[
F(\bar{x},y) \geq \limsup_{k \tendsto \infty} F(x^{k},y) \geq 0, \quad \forall y \in K.
\]
So, we have $\bar{x} \in \mathcal{E}(K,F)$ and hence the closedness of $\mathcal{E}(K,F)$.

Now, assume that $x \mapsto F(x,y)$ is quasi-concave for every fixed $y \in K$. Let $\hat{x},\hat{y} \in \mathcal{E}(K,F)$ and $t \in [0,1]$. Then, we have
\[
F(\gamma_{\hat{x},\hat{y}}(t),y) \geq \min\{F(\hat{x},y),F(\hat{y},y)\} \geq 0, \quad \forall y \in K,
\]
which implies the convexity of $\mathcal{E}(K,F)$.
\end{proof}

Next, we deduce an existence theorem for a compact convex domain.
\begin{cor}\label{cor:compactcase}
Suppose that $K \subset X$ is compact convex, and $F : K \times K \to \R$ is a bifunction satisfying the following properties:
\begin{enumerate}[label=(A\arabic*)]
\item $F(x,x) \geq 0$ for each $x \in K$.
\item For every $x \in K$, the set $\{y \in K \;;\; F(x,y) < 0\}$ is convex.
\item For every $y \in K$, the function $x \mapsto F(x,y)$ is usc.
\end{enumerate}
Then, the problem $EP(K,F)$ has a solution and the solution set $\mathcal{E}(K,F)$ is closed. In addition, if the function $x \mapsto F(x,y)$ is quasi-concave for every fixed $y \in K$, then $\mathcal{E}(K,F)$ is convex.
\end{cor}
\begin{proof}
Apply Theorem \ref{thm:mainexistence} with $L = K$ in the hypothesis \ref{asmp:4}.
\end{proof}

\subsection{Dual Equilibrium Problem}

It is essential to also introduce here the duality to the problem $EP(K,F)$ in the sense of Minty. To clarify the terminology, we shall sometimes refer to $EP(K,F)$ as the \emph{primal problem}. The \emph{dual equilibrium problem} to $EP(K,F)$, denoted by $EP^{\ast}(K,F)$, is given as follows:
\[
\text{Find a point $\bar{x} \in K$ such that $F(y,\bar{x}) \leq 0$ for every $y \in K$.}	\tag*{$EP^{\ast}(K,F)$}
\]
The solution to $EP^{\ast}(K,F)$ will also be called the \emph{dual solution to $EP(K,F)$}, and the set of such dual solutions is then denoted by $\mathcal{E}^{\ast}(K,F)$.

\begin{dfn}
$F$ is called \emph{monotone} if $F(x,y) + F(y,x) \leq 0$ for all $x,y \in K$.
\end{dfn}

If $F$ is a monotone bifunction, we immediately have the inclusion $\mathcal{E}(K,F) \subset \mathcal{E}^{\ast}(K,F)$. To obtain the converse, we need additional assumptions on the bifunction $F$.

\begin{prop}\label{prop:primal=dual}
Suppose that $F :K \times K \to \R$ is a bifunction.
\begin{enumerate}[label=(\roman*)]
\item\label{cdn:primal->dual} If $F$ is monotone, then $\mathcal{E}(K,F) \subset \mathcal{E}^{\ast}(K,F)$.
\item\label{cdn:dual->primal} If \ref{asmp:1}, \ref{asmp:3} holds, and $F$ is convex in the second variable, then $\mathcal{E}^{\ast}(K,F) \subset \mathcal{E}(K,F)$.
\end{enumerate}
In particular, if $F$ is monotone, convex in the second variable, and satisfies
\ref{asmp:1} and \ref{asmp:3}, then $\mathcal{E}^{\ast}(K,F) = \mathcal{E}(K,F)$.
\end{prop}
\begin{proof}
\ref{cdn:primal->dual} Let $\bar{x} \in \mathcal{E}(K,F)$. The monotonicity of $F$ then gives $F(y,\bar{x}) \leq - F(\bar{x},y) \leq 0$, so that $\bar{x} \in \mathcal{E}^{\ast}(K,F)$.

\ref{cdn:dual->primal} Suppose that $\bar{x} \in \mathcal{E}^{\ast}(K,F)$. Let $y \in K$ be arbitrary. Take into account \ref{asmp:1}, we have
\[
0 \leq F(\gamma_{\bar{x},y}(t),\gamma_{\bar{x},y}(t)) \leq (1-t)F(\gamma_{\bar{x},y}(t),\bar{x}) + t F(\gamma_{\bar{x},y}(t),y) \leq F(\gamma_{\bar{x},y}(t),y),
\]
for any $t \in [0,1]$. In view of \ref{asmp:3}, we have $0 \leq \limsup_{t \tendsto 0} F(\gamma_{\bar{x},y}(t),y) \leq F(\bar{x},y)$. This gives $\bar{x} \in \mathcal{E}(K,F)$.
\end{proof}


\section{Resolvents of Bifunctions}\label{sec:resolvents}

The resolvent is a fundamental notion in regularization of certain variational problems. It is understood as the solution set of particular regularized problem from the original one. Here, we introduce the resolvent of a bifunction $F$ in relation to the equilibrium problem $EP(K,F)$.

For simplicity, we shall adopt the following perturbation $\tilde{F}$ of a given bifunction $F$. That is, given a bifunction $F : K \times K \to \R$, $\bar{x} \in X$, we define the function $\tilde{F}_{\bar{x}} : K \times K \to \R$ by
\[
\tilde{F}_{\bar{x}} (x,y) := F(x,y) - \pair{\Vec{x\bar{x}}}{\Vec{xy}}, \quad \forall x,y \in K.
\]
We are now ready to construct the resolvent of a bifunction $F$ as the unique equilibrium of a perturbed bifuntion.

\begin{dfn}[\cite{Kimura-Kishi}]
Suppose that $K \subset X$ is closed convex, and $F : K \times K \to \R$. The \emph{resolvent} of $F$ is the mapping $J_{F} : X \multimap K$ defined by
\[
J_{F} (x) := \mathcal{E}(K,\tilde{F}_{x}) = \{z \in K \,|\, F(z,y) - \pair{\Vec{zx}}{\Vec{zy}} \geq 0,\, \forall y \in K\}, \quad \forall x \in X.
\]
\end{dfn}

The question well-definedness of $J_{F}$ is important. Kimura and Kishi \cite{Kimura-Kishi} showed a well-definedness result under the CHFP assumption. The use of CHFP assumption is to trigger the KKM principle of Niculescu and Roven\c{t}a \cite{MR2561730}. However, we can replace such KKM principle of Niculescu and Roven\c{t}a \cite{MR2561730} with our Theorem \ref{lem:KKM} and follows the same proof from \cite{Kimura-Kishi} to obtain the well-definedness of $J_{F}$. Thus, we obtain Kimura and Kishi's theorem without CHFP.

\begin{thm}
Suppose that $F$ has the following properties:
\begin{enumerate}[label=(\roman*)]
\item $F(x,x) = 0$ for all $x \in K$.
\item $F$ is monotone.
\item For each $x \in K$, $y \mapsto F(x,y)$ is convex and lsc.
\item For each $y \in K$, $F(x,y) \geq \limsup_{t \dto 0} F(\gamma_{x,z}(t),y)$ for all $x,z \in K$.
\end{enumerate}
Then $\dom(J_{F}) = X$ and $J_{F}$ is single-valued.
\end{thm}

Note that if we drop the assumption that $y \mapsto F(x,y)$ is lsc, we get the following result by imposing the flatness and local compactness.

\begin{prop}
Suppose that hypotheses \ref{asmp:1}, \ref{asmp:3}, and \ref{asmp:4} from Theorem \ref{thm:mainexistence} are satisfied. Suppose also that the following additional assumptions hold:
\begin{enumerate}[label=(\roman*)]
\item $F$ is monotone.
\item for any fixed $x \in K$, $y \mapsto F(x,y)$ is convex.
\end{enumerate}
If a nonempty convex subset $N \subset X$ is flat and locally compact, then the resolvent $J_{F}$ is defined for all $x \in N$. That is, $N \subset \dom(J_{F})$.
\end{prop}
\begin{proof}
Fix $\tilde{x} \in N$, we define a bifunction $g_{\tilde{x}} : K \times K \to \R$ by
\[
g_{\tilde{x}} (z,y) := F(z,y) - \pair{\Vec{z\tilde{x}}}{\Vec{zy}}, \quad \forall z,y \in K.
\]
Then $g_{\tilde{x}}(z,z) \geq 0$ for all $z \in K$, and $z \mapsto g_{\tilde{x}}(z,y)$ is usc for any fixed $y \in K$. Since $K$ is flat, the mapping $y \mapsto \pair{\Vec{z\tilde{x}}}{\Vec{zy}}$ is affine by \ref{prop:affine} of Lemma \ref{prop:twopair}. Therefore, the set $\{y \in K \;;\; g_{\tilde{x}}(z,y) < 0\}$ is convex for each $z \in K$.

Since all assumptions of Theorem \ref{thm:mainexistence} are satisfied, $F$ has an equilibrium $\bar{x} \in K$. Now, using the monotonicity of $F$ and \ref{prop:twopair1} of Lemma \ref{prop:twopair}, we get
\begin{align*}
g_{\tilde{x}}(z,\bar{x}) &= F(z,\bar{x}) - \pair{\Vec{z\tilde{x}}}{\Vec{z\bar{x}}} \\
&\leq -F(\bar{x},z) + \pair{\Vec{\bar{x}\tilde{x}}}{\Vec{\bar{x}z}} - \rho^{2}(\bar{x},z) \\
&\leq (\rho(\bar{x},\tilde{x}) - \rho(\bar{x},z))\rho(\bar{x},z).
\end{align*}
Set $L := K \cap \cl A(\bar{x};\rho(\bar{x},\tilde{x}))$, where $\cl(\cdot)$ denotes the closure operator. Since $K$ is locally compact and $L$ is closed and bounded, we get the compactness of $L$. Moreover, we have $\bar{x} \in L$ and $g_{\tilde{x}}(z,\bar{x}) < 0$ for all $z \in K \setminus L$. In fact, we have just showed that $g_{\tilde{x}}$ verifies all assumptions of Theorem \ref{thm:mainexistence}, and therefore $g_{\tilde{x}}$ has an equilibrium $\bar{z} \in K$. Equivalently, we have proved that $\bar{z} \in J_{F}(\tilde{x})$. Since $\tilde{x} \in N$ is chosen at arbitrary, the desired result is attained.
\end{proof}


The following proposition gives several valuable properties for a resolvent of a monotone bifunction.

\begin{prop}\label{prop:resolventPROP}
Suppose that $F$ is monotone and $\dom(J_{F}) \neq \emptyset$. Then, the following properties hold.
\begin{enumerate}[label=(\roman*)]
\item\label{cdn:singlevalued} $J_{F}$ is single-valued.
\item\label{cdn:NX} If $\dom(J_{F}) \supset K$, then $J_{F}$ is nonexpansive restricted to $K$.
\item\label{cdn:FPcharact} If $\dom(J_{\mu f}) \supset K$ for any $\mu > 0$, then $Fix(J_{F}) = \mathcal{E}(K,F)$.
\end{enumerate}
\end{prop}
\begin{proof}
\ref{cdn:singlevalued} Let $x \in \dom(J_{F})$ and suppose that $z,z' \in J_{F}(x)$. So, we get
\[
\left\{\begin{array}{l}
F(z,z') \geq \pair{\Vec{zx}}{\Vec{zz'}}, \smallskip\\
F(z',z) \geq \pair{\Vec{z'x}}{\Vec{z'z}}.
\end{array}\right.
\]
By summing up the two inequalities, applying the monotonicity of $F$ and \ref{prop:twopair1} of Lemma \ref{prop:twopair}, we obtain
\[
0 \geq F(z,z') + F(z',z) \geq \pair{\Vec{zx}}{\Vec{zz'}} + \pair{\Vec{z'x}}{\Vec{z'z}} = \rho^{2}(z,z').
\]
This shows $z = z'$.


\ref{cdn:NX} Let $x,y \in K$. By the definition of $J_{F}$, we have
\[
\left\{
\begin{array}{l}
F(J_{F}(x),J_{F}(y)) - \pair{\Vec{J_{F}(x)x}}{\Vec{J_{F}(x)J_{F}(y)}} \geq 0, \medskip\\
F(J_{F}(y),J_{F}(x)) - \pair{\Vec{J_{F}(y)y}}{\Vec{J_{F}(y)J_{F}(x)}} \geq 0.
\end{array}
\right.
\]
Summing the two inequalities above yields
\begin{align*}
0 &\geq \pair{\Vec{J_{F}(x)x}}{\Vec{J_{F}(x)J_{F}(y)}} + \pair{\Vec{J_{F}(y)y}}{\Vec{J_{F}(y)J_{F}(x)}} \\
&= \left[ \rho^{2}(J_{F}(x),x) + \rho^{2}(J_{F}(x),J_{F}(y)) - \rho^{2}(x,J_{F}(y)) \right] \\
& \qquad + \left[ \rho^{2}(J_{F}(y),y) + \rho^{2}(J_{F}(x),J_{F}(y)) - \rho^{2}(y,J_{F}(x)) \right].
\end{align*}
Rearraging terms in the above inequality and apply Proposition \ref{prop:ql-ineq}, we get
\small\begin{align*}
\rho^{2}(J_{F}(x),J_{F}(y)) &\leq \frac{1}{2} \left[ \rho^{2}(x,J_{F}(y)) - \rho^{2}(y,J_{F}(x)) - \rho^{2}(x,J_{F}(x)) - \rho^{2}(y,J_{F}(y)) \right] \\
&\leq \rho(x,y)\rho(J_{F}(x),J_{F}(y)).
\end{align*}\normalsize
This shows the nonexpansivity of $J_{F}$.


\ref{cdn:FPcharact} Let $x \in K$. Observe that
\begin{align*}
x \in Fix(J_{F}) &\iff x = J_{F}(x)\\
&\iff F(x,y) - \pair{\Vec{xx}}{\Vec{xy}} \geq 0,\, \forall y \in K\\
&\iff F(x,y) \geq 0,\, \forall y \in K \\
&\iff x \in \mathcal{E}(K,F). \tag*\qedhere
\end{align*}
\end{proof}

\section{Bifunction Rosolvents and Other Resolvents}\label{sec:resolv-compare}

In this section, we consider two special cases of equilibrium problem, where the resolvent associated to a bifunction defined in Section \ref{sec:resolvents} conincides with the resolvents designed for solving different variational problems. In particular, we deduce that our resolvent reduces to the Moreau-Yosida resolvent for a convex functional, and also to resolvents corresponding maximal monotone operators.

\subsection{Resolvents of convex functionals}

Always assume that $g : X \to \R$ is convex and lsc. The proximal of $g$ is the operator $\prox_{g} : X \to X$ given by
\[
\prox_{g} (x) := \argmin_{y \in X} \left[g(y) + \frac{1}{2}\rho^{2}(y,x) \right], \quad \forall  x \in X.
\]
Also, recall (from \cite{MR3374067,MR3691338}) that a subdifferential of $g$ is the set-valued vector field $\partial g : X \multimap X^{\ast}$ given by
\[
\partial g(x) := \{\gamma \in T_{x}X \;|\; g(y) \geq g(x) + \pair{\gamma}{\Vec{xy}}, \,\forall y \in X\},
\]
for $x \in X$. It is shown that $\dom(\partial g)$ is dense in $\dom(g)$ (in this case, $\dom(g) = X$) \cite{MR3374067}. Moreover, we have the following lemma.
\begin{lem}[\cite{MR3691338}]\label{lem:resolvent-convexfunctional}
Given $w,z \in X$, then $\Vec{zw} \in \partial g (z)$ if and only if $z = \prox_{g}(w)$.
\end{lem}

It is immediate that minimizing this functional $g$ is a particular equilibrium problem. We make the following statement explicit only for the completeness.
\begin{lem}\label{lem:fg}
Define $F_{g} : K \times K \to \R$ by
\begin{equation}\label{eqn:fg}
F_{g}(x,y) := g(y) - g(x), \quad \forall x,y \in K.
\end{equation}
Then, we have $\mathcal{E}(K,F_{g}) = \argmin_{K} g$ and $J_{F_{g}} = \prox_{g}$. Moreover, we have $\dom(\prox_{g}) = X$.
\end{lem}
\begin{proof}
The fact that $\mathcal{E}(K,F_{g}) = \argmin_{K} g$ is obvious. Now, Lemma \ref{lem:resolvent-convexfunctional} implies the following relations:
\begin{align*}
z = J_{F_{g}}(x) &\iff F_{g}(z,y) - \pair{\Vec{zx}}{\Vec{zy}} \geq 0, \quad\forall y \in K \\
&\iff g(y) \geq g(z) + \pair{\Vec{zx}}{\Vec{zy}}, \quad\forall y \in K \\
&\iff \Vec{zx} \in \partial g(z) \\
&\iff z = \argmin_{y \in X} \left[g(y) + \frac{1}{2}\rho^{2}(y,x)\right] = \prox_{g}(x)
\end{align*}
The fact that $\dom(\prox_{g}) = X$ was proved in \cite{MR1360608,MR1651416}.
\end{proof}

\subsection{Resolvents of monotone operators}

Following \cite{MR3691338}, the set-valued operator $A : X \multimap X^{\ast}$ is called a \emph{vector field} if $A(x) \subset T_{x}X$ for every $x \in X$. Moreover, it is said to be \emph{monotone} if $\pair{x^{\ast}}{\Vec{xy}} \leq -\pair{y^{\ast}}{\Vec{yx}}$ for every $(x,x^{\ast}),(y,y^{\ast}) \in \grp(A)$. It is said to be \emph{maximally monotone} if $A$ is monotone and $\grp(A)$ is not properly contained in a graph of another monotone vector field.

We shall make an observation in this section that the following \emph{stationary problem} (or \emph{zero point problem}) associated to a monotone set-valued vector field $A : X \multimap X^{\ast}$:
\[\label{eqn:singularity}
\text{Find a point $\bar{x} \in K$ such that $\0 \in A(\bar{x})$,}	\tag{$A^{-1}\0$}
\]
can be viewed in terms of an equilibrium problem for some bifuntion $F_{A}$. The solution set is naturally expressed by $A^{-1}\0$.


Beforehand, let us give a restatement of a result given by \cite{MR3691338} in the language of this paper. In fact, this states the equivalence between the singularity priblem and the Minty variational inequality of $A$.

\begin{prop}\label{prop:sin=prim=dual}
Suppose that $A : X \multimap X^{\ast}$ is a monotone vector field with $\dom(A) = X$. Define $F_{A} : X \times X \to \R$ by
\begin{equation}\label{eqn:fS}
F_{A}(x,y) := \sup_{\xi \in A(x)} \pair{\xi}{\Vec{xy}}, \quad \forall x,y \in X,
\end{equation}
Then, $A^{-1}\0 = \mathcal{E}(X,F_{A}) = \mathcal{E}^{\ast}(X,F_{A})$.
\end{prop}
\begin{proof}
First, we note that the finiteness of $F_{A}$ is guaranteed readily from the monotonicity of $A$. Moreover, the fact that $\mathcal{E}(X,F_{A}) \subset \mathcal{E}^{\ast}(X,F_{A}) = A^{-1}\0$ follows from the monotonicity of $F_{A}$, \cite[Lemma 3.5]{MR3691338}, and \cite[Lemma 3.6]{MR3691338}. It therefore suffices to show that $A^{-1}\0 \subset \mathcal{E}(X,F_{A})$. So, let $\0 \in A(\bar{x})$. Then
\[
0 = \pair{\Vec{\bar{x}\bar{x}}}{\Vec{\bar{x}y}} \leq \sup_{\bar{\xi} \in A(\bar{x})} \pair{\bar{\xi}}{\Vec{\bar{x}y}} = F_{A}(\bar{x},y), \quad \forall y \in X,
\]
meaning that $\bar{x} \in \mathcal{E}(X,F_{A})$.
%
\end{proof}

Now that we have the equivalence between the three variational problems, we continue to show that their resolvents coincide, provided that $A$ is maximally monotone. Recall from \cite{MR3691338} that the resolvent operator of $A$, denoted by $R_{A}$, is defined by $R_{A}(x) := \{z \in X \,|\, \Vec{zx} \in A(z)\}$, and is single-valued.
\begin{prop}\label{prop:J=W}
Suppose that $A : X \multimap X^{\ast}$ is a maximally monotone vector field with $\dom(A) = X$, and $F_{A}$ is defined by \eqref{eqn:fS}.  Then, $J_{F_{A}} = R_{A}$.
\end{prop}
\begin{proof}
Let $x \in X$ be given. Then, we have
\begin{align}\label{eqn:JS=JfS}
z = R_{A}(x) &\iff \Vec{zx} \in A(z) \nonumber\\
&\;\implies \pair{\Vec{zx}}{\Vec{zy}} \leq \sup_{\eta \in A(z)} \pair{\eta}{\Vec{zy}}, \;\forall y \in X \\
&\iff z \in J_{F_{A}}(x). \nonumber
\end{align}
The converse of \eqref{eqn:JS=JfS} follows from the maximal monotonicity of $A$.
\end{proof}

\section{Proximal Algorithms}\label{sec:convergence}

It is very natural to ask about the proximal algorithm after defining a proper resolvent operator. Let us officially define the proximal algorithm for a bifunction $F : K \times K \to \R$ with $\dom(\mu F) \supset K$ for all $\mu > 0$. Let $(\lambda_{k}) \subset (0,\infty)$ be the step-size sequence. The \emph{proximal algorithm} with step sizes $(\lambda_{k})$ started at an initial guess $x^{0} \in K$ is the sequence $(x^{k}) \subset K$ generated by
\[\label{eqn:proximal}
x^{k} := J_{\lambda_{k}F}(x^{k-1}), \quad \forall k \in \N.	\tag{Prox}
\]

\subsection{Convergence in functional values}
Before we continue any further, let us give a small remark on the error measurement of this proximal algorithm. For instance, at each $k \in \N$, we can use the definition of the resolvent $J_{F}$ and the Cauchy-Schwarz inequality to deduce:
\begin{equation}\label{eqn:estimate0}
\inf_{y \in K} F(x^{k+1},y) \geq -\frac{1}{\lambda_{k}}\rho(x^{k+1},x^{k})\sup_{y \in K} \rho(x^{k+1},y),
\end{equation}
which is useful when as a stopping criterion for the convergence, especially when the rate of asymptotic regularity is known. Note that the above estimate \eqref{eqn:estimate0} will later be made sharp and precise to guarantee the convergence of the proximal algorithm (see Theorem \ref{thm:convergence1} and \ref{thm:convergence2}, for instance). If the asymptotic regularity does not hold or the rate is unknown but $K$ is bounded, we otherwise have
\begin{equation}\label{eqn:estimate1}
\inf_{y \in K} F(x^{k+1},y) \geq - \frac{1}{\lambda_{k}}\diam(K)^{2},
\end{equation}
which leads directly to the convergence of functional value presented in the next theorem.
\begin{thm}
Suppose that $\dom(J_{\mu F}) \supset K$ for all $\mu > 0$ and $K$ is bounded closed convex. If $\lambda_{k} \tendsto \infty$, then the proximal algorithm generates a sequence $(x^{k}) \subset K$ such that $\lim_{k} F(x^{k},y) = 0$ for any $y \in K$.
\end{thm}
\begin{proof}
Follows from \eqref{eqn:estimate1}.
\end{proof}

%
%
%

\subsection{Convergence of Proximal Algorithms}

Here, we provide a convergence theorem for proximal algorithm \eqref{eqn:proximal} for step sizes $(\lambda_{k})$. To prove the convergence, we first show the following `obtuse angle' property, which gives relationship between an arbitrary point $\bar{x} \in X$, the perturbed equilibrium $\tilde{x} = J_{\mu F} (\bar{x}) \in \mathcal{E}(K,\widetilde{\mu F}_{\bar{x}})$, and the exact equilibrium $x^{\ast} \in \mathcal{E}(K,F)$.

\begin{lem}\label{lem:obtuse}
Assume that $F$ is monotone. Let $\bar{x} \in X$, $\mu > 0$, $\tilde{x} \in \mathcal{E}(K,\widetilde{\mu F}_{\bar{x}})$ and $x^{\ast} \in \mathcal{E}(K,F)$, then $\pair{\Vec{\tilde{x}\bar{x}}}{\Vec{\tilde{x}x^{\ast}}} \leq 0$.
\end{lem}
\begin{proof}
Since $\tilde{x} \in \mathcal{E}(K,\tilde{F}_{\bar{x}})$, we have
\[
0 \leq \widetilde{\mu F}_{\bar{x}} (\tilde{x},x^{\ast}) = \mu F(\tilde{x},x^{\ast}) - \pair{
\Vec{\tilde{x}\bar{x}}}{\Vec{\tilde{x}x^{\ast}}},
\]
which implies that $\pair{
\Vec{\tilde{x}\bar{x}}}{\Vec{\tilde{x}x^{\ast}}} \leq \mu F(\tilde{x},x^{\ast})$. Now, since $x^{\ast} \in \mathcal{E}(K,F)$ and $F$ is monotone, we get $F(y,x^{\ast}) \leq 0,\, \forall y \in K$, and particularly $F(\tilde{x},x^{\ast}) \leq 0$. We therefore have $\pair{
\Vec{\tilde{x}\bar{x}}}{\Vec{\tilde{x}x^{\ast}}} \leq 0$.
\end{proof}

\begin{thm}\label{thm:convergence1}
Suppose that $F$ is monotone with $\mathcal{E}(K,F) \neq \emptyset$, $\Delta$-usc in the first variable, and that $\dom(J_{\mu F}) \supset K$ for all $\mu > 0$. Let $(\lambda_{k})$ be bounded away from $0$. Then the proximal algorithm \eqref{eqn:proximal} is $\Delta$-convergent to an element in $\mathcal{E}(K,F)$ for any initial start $x^{0} \in K$.
\end{thm}
\begin{proof}
Let $x^{0} \in K$ be an initial start and let $x^{\ast} \in \mathcal{E}(K,F)$. We can simply see that
\begin{align*}
\rho(x^{\ast},x^{k+1}) = \rho(J_{\lambda_{k}F}(x^{\ast}),J_{\lambda_{k}F}(x^{k})) \leq \rho(x^{\ast},x^{k}),
\end{align*}
which implies that $(x^{k})$ is Fej\'er convergent with respect to $\mathcal{E}(K,F)$. In view of Proposition \ref{prop:Fejer}, the real sequence $\rho(x^{k},x^{\ast})$ is bounded, and hence converges to some $\xi \geq 0$. Lemma \ref{lem:obtuse} implies that
\[
\rho^{2}(x^{k+1},x^{k}) \leq \rho^{2}(x^{k},x^{\ast}) - \rho^{2}(x^{k+1},x^{\ast}).
\]
Passing $k \tendsto \infty$, we get $\lim_{k} \rho(x^{k+1},x^{k}) = 0$.

Now, suppose that $\hat{x} \in K$ is a $\Delta$-accumulation point of the sequence $(x^{k})$, and also $(x^{k_{j}}) \subset (x^{k})$ a subsequence such that $x^{k_{j}} \overset{\Delta}{\tendsto} \hat{x}$. Let $y \in K$. By the construction \eqref{eqn:proximal}, we have the following inequalities for any $j \in \N$:
\begin{equation}\label{eqn:preconvergence}
F(x^{k_{j}},y) \geq \frac{1}{\lambda_{k_{j}}}\pair{\Vec{x^{k_{j}}x^{k_{j}-1}}}{\Vec{x^{k_{j}}y}} \geq - \frac{1}{\lambda_{k_{j}}}\rho(x^{k_{j}},x^{k_{j}-1})\rho(x^{k_{j}},y).
\end{equation}
Recall that $(x^{k})$ is bounded (in view of Proposition \ref{prop:Fejer}) and $(\lambda_{k})$ is bounded away from $0$. Then \eqref{eqn:preconvergence} gives
\begin{equation}\label{eqn:preconvergence2}
F(x^{k_{j}},y) \geq -M\rho(x^{k_{j}},x^{k_{j}-1}),
\end{equation}
for some $M > 0$.
Passing $j \tendsto \infty$ in \eqref{eqn:preconvergence2} and apply the $\Delta$-upper semiconinuity of $F(\cdot,y)$, we obtain
\[
F(\hat{x},y) \geq \limsup_{j \tendsto \infty} F(x^{k_{j}},y) \geq -M\lim_{j \tendsto \infty} \rho(x^{k_{j}},x^{k_{j}-1}) = 0.
\]
Since $y \in K$ is chosen arbitrarily, we conclude that $\hat{x} \in \mathcal{E}(K,F)$. Therefore, every $\Delta$-accumulation point of $(x^{k})$ solves $EP(K,F)$. By Proposition \ref{prop:Fejer}, the sequence $(x^{k})$ is $\Delta$-convergent to an element in $\mathcal{E}(K,F)$.
\end{proof}

\begin{cor}[\cite{MR3047087,MR3691338}]
Suppose that $g : X \to \R$ is convex and lsc, and that $\argmin g \neq \emptyset$. Then, the proximal algorithm given by
\[
\left\{
\begin{array}{l}
x^{0} \in X, \medskip\\
x^{k} := \prox_{\lambda_{k}g}(x^{k-1}), \quad \forall k \in \N,
\end{array}
\right.
\]
is $\Delta$-convergent to a minimizer of $g$, whenever $(\lambda_{k}) \subset (0,\infty)$ is bounded away from $0$.
\end{cor}
\begin{proof}
Consider the bifunction $F_{g}$ as defined in Lemma \ref{lem:fg}. The monotonicity of $F_{g}$ is immediate, and the fact that $F_{g}$ is $\Delta$-usc in the first variable follows by applying Lemma \ref{lem:closed+convex=Deltaclosed} to the epigraph of $F_{g}(x,\cdot)$ at each $x \in X$. The convergence is then a consequence of Theorem \ref{thm:convergence1} applied to $F_{g}$, where the remaining requirements of $F_{g}$ follows from Lemma \ref{lem:fg}.
\end{proof}

Let us look closer at the assumptions of Theorem \ref{thm:convergence1}, as well as the bifunction $F_{A}$ given by \eqref{eqn:fS} associated to some maximal monotone vector field $A$. One may notice that the $\Delta$-usc of in the first variable of $F_{A}$ is irrelevant. This motivates us to find the right mechanism to link between proximal algorithms for equilibrium problems and for monotone vector fields. Fortunately enough, we can deduce another convergence criteria that is capable to include \cite[Theorem 5.2]{MR3691338} as a certain special case. To do this, we need the following notion of skewed $\Delta$-semicontinuity of $F$.
\begin{dfn}
A bifunction $F : K \times K \to \R$ is said to be \emph{skewed $\Delta$-upper semicontinuous} (for short, \emph{skewed $\Delta$-usc}) if $-F(y,x^{\ast}) \geq \limsup_{k} F(x^{k},y)$ for all $y \in K$, whenever $(x^{k}) \subset K$ is $\Delta$-convergent to $x^{\ast} \in K$.
\end{dfn}
\begin{rmk}
It is clear that if $F$ is monotone and $\Delta$-usc in the first variable, then $F$ is skewed $\Delta$-usc.
\end{rmk}

The next lemma shows that we can deduce that $F_{A}$ is skewed $\Delta$-usc for a monotone vector field $A : X \multimap X^{\ast}$.
\begin{lem}\label{lem:fB-skewedUSC}
Suppose that $A : X \multimap X^{\ast}$ is a monotone vector field with $\dom(A) = X$, and $F_{A}$ is defined by \eqref{eqn:fS}. If $X$ is reflexive, then $F_{A}$ is skewed $\Delta$-usc.
\end{lem}
\begin{proof}
Suppose that $(x^{k}) \subset K$ is a $\Delta$-convergent to $x^{\ast} \in K$, and $y \in K$ be arbitrary. Then, by the monotonicity of $F_{A}$, we have
\begin{align}
\limsup_{k \tendsto \infty} F_{A}(x^{k},y) &\leq \limsup_{k \tendsto \infty} [ -F_{A}(y,x^{k})] \nonumber\\
&= \limsup_{k \tendsto \infty} [- \sup_{\nu \in A(y)} \pair{\nu}{\Vec{yx^{k}}}] \nonumber\\
&\leq \limsup_{k \tendsto \infty} [- \pair{\nu_{0}}{\Vec{yx^{k}}}], \label{eqn:upperest}
\end{align}
for any $\nu_{0} \in A(y)$. Suppose that $\nu_{0} = \delta\Vec{yu}$ for some $u \in X$ and $\delta > 0$. Then, we have
\begin{equation}\label{eqn:extract1}
\pair{\nu_{0}}{\Vec{yx^{k}}} = \pair{\nu_{0}}{\Vec{yx^{\ast}}} + \pair{\Vec{x^{\ast}u}^{\delta}}{\Vec{x^{\ast}x^{k}}} - \pair{\Vec{x^{\ast}y}^{\delta}}{\Vec{x^{\ast}x^{k}}}.
\end{equation}
Combine \eqref{eqn:upperest}, \eqref{eqn:extract1}, and take into account the reflexivity of $X$ and the $\Delta$-convergence of $(x^{k})$, we obtain
\[
\limsup_{k \tendsto \infty} F_{A}(x^{k},y) \leq - \pair{\nu_{0}}{\Vec{yx^{\ast}}}.
\]
Since this is true for any $\nu_{0} \in A(y)$, we get
\[
\limsup_{k \tendsto \infty} F_{A}(x^{k},y) \leq \inf_{\nu \in T(y)} [- \pair{\nu}{\Vec{yx^{\ast}}}] = - \sup_{\nu \in T(y)} \pair{\nu}{\Vec{yx^{\ast}}} = -F_{A}(y,x^{\ast}). \tag*{\qedhere}
\]
\end{proof}

Now that we have the motivations for skewed $\Delta$-semicontinuity, we shall now give another convergence theorem based on this new notion of continuity. Recall that the coincidence of the primal and dual solutions of an equilibrium problem, as appeared in the following theorem, can be referenced from either Propositions \ref{prop:primal=dual} or \ref{prop:sin=prim=dual}.

\begin{thm}\label{thm:convergence2}
Suppose that $F$ is a monotone bifunction such that $\mathcal{E}(K,F) = \mathcal{E}^{\ast}(K,F) \neq \emptyset$, skewed $\Delta$-usc, and that $\dom(J_{\mu F}) \supset K$ for all $\mu > 0$. Let $(\lambda_{k})$ be bounded away from $0$. Then the proximal algorithm \eqref{eqn:proximal} is $\Delta$-convergent to an element in $\mathcal{E}(K,F)$ for any initial start $x^{0} \in K$.
\end{thm}
\begin{proof}
With similar proof lines to Theorem \ref{thm:convergence1}, we can also be able to obtain \eqref{eqn:preconvergence2}. Since $F$ is skewed $\Delta$-usc, we obtain the following inequalities by passing $j \tendsto \infty$:
\[
-F(y,\hat{x}) \geq \limsup_{j \tendsto \infty} F(x^{k_{j}},y) \geq 0,
\]
for each $y \in K$. This means $\hat{x} \in \mathcal{E}^{\ast}(K,F) = \mathcal{E}(K,F)$, by hypothesis. Therefore, every $\Delta$-accumulation point lies within $\mathcal{E}(K,F)$. Apply Proposition \ref{prop:Fejer} to conclude that $(x^{k})$ is $\Delta$-convergent to an element in $\mathcal{E}(K,F)$.
\end{proof}

\begin{cor}[\cite{MR3691338}]
Suppose that $A : X \multimap X^{\ast}$ is a monotone vector field with $\dom(A) = X$ and $\dom(\mu R_{A}) = X$ for all $\mu > 0$, and that $A^{-1}\0 \neq \emptyset$. Assume that $X$ is reflexive. Then, the proximal algorithm defined by
\[
\left\{
\begin{array}{l}
x^{0} \in X, \medskip\\
x^{k} := R_{\lambda_{k}A}(x^{k-1}), \quad \forall k \in \N,
\end{array}
\right.
\]
is $\Delta$-convergent to an element in $A^{-1}\0$, whenever $(\lambda_{k}) \subset (0,\infty)$ is bounded away from $0$.
\end{cor}
\begin{proof}
It was proved in \cite{MR3691338} that the surjectivity condition of $A$ implies the maximal monotonicity and also gives $\dom(J_{\mu F_{A}}) = X$ (for all $\mu > 0$) in light of Proposition \ref{prop:J=W}, where $F_{A}$ defined in \eqref{eqn:fS}. The monotonicity of $F_{A}$ follows from the monotonicity of $A$, while the skewed $\Delta$-semicontinuity of $F_{A}$ is proved in Lemma \ref{lem:fB-skewedUSC}. By the Proposition \ref{prop:sin=prim=dual}, it is feasible to apply Theorem \ref{thm:convergence2} to $F_{A}$ and obtain the desired convergence result.
\end{proof}

\begin{rmk}
Amounts to \eqref{eqn:estimate0} and the Fej\'{e}r convergence of the proximal algorithm presented in the proofs of Theorems \ref{thm:convergence1} and \ref{thm:convergence2}, it can be seen that we can make the convergence of functional values $\inf_{y \in K} F(x^{k},y) \tendsto 0$ arbitrarily fast by the choosing appropriate step sizes. However, this does not ensure the speed of convergence of the sequence $(x^{k})$ even if the convergence is strong (i.e., in the metric topology). In particular, the speed enhancement of proximal algorithm for singularity problem using metric regularity conditions was developed in \cite{MR3691338}.
\end{rmk}
%
%
%
%

\section*{Conclusion and Remarks}

We have provided a complete treatment of equilibrium problem situated in Hadamard spaces, from existence to approximation algorithm. We were also able to prove the KKM principle without additional assumptions, as opposed to earlier results in the literature. The existence criteria for an equilibrium is deduced under standard assumptions, which is very natural in this subject area.

In the approximation, we investigated a priori on the resolvent operator for a given bifunction. Main results here are that the resolvent operator is single-valued and firmly nonexpansive. We then define the proximal algorithm by iterating the resolvents of different bifurcating parameters. Several convergence criteria of the algorithm were proposed, including the one that involves skewed continuity, where we introduced here for the first time. Note again that the corresponding functional values converging to $0$ can be achieved arbitrarily fast, but this does not imply the same for proximal algorithm itself.

Let us conclude this paper with some open questions whose answers might largely improve the applicability of the results in this present paper.
\begin{ques}
Whether or not we can improve the following condition: $\dom(J_{\mu F}) \supset K$ for all $\mu > 0$, in order to obtain similar results regarding resolvent operators and proximal algorithms?
\end{ques}
\begin{ques}
Is it possible to drop the surjectivity condition above and still obtain the nonexpansivity of $J_{F}$ (see \ref{cdn:NX} in Proposition \ref{prop:resolventPROP})?
\end{ques}

\section*{Acknowledgements}

The second author was supported jointly by the Thailand Research Fund (TRF) and King Mongkuts University of Technology Thonburi (KMUTT) through the Royal Golden Jubilee Ph.D. Program (Grant No. PHD/0045/2555).


\end{document}